\documentclass[reqno]{amsart}

\usepackage[utf8]{inputenc}
\usepackage[T1]{fontenc}
\usepackage{amscd,amssymb,graphicx,subfigure,mathtools}
\usepackage{hyperref}
\usepackage[margin=1.25in]{geometry}
\usepackage{enumerate}

\usepackage{pgf,tikz} 
\usetikzlibrary{arrows,shapes}
\usetikzlibrary{decorations.pathreplacing}
\usepackage{tkz-berge}
\usepackage{mathdots}
\usepackage{epstopdf}

\usepackage{todonotes} 
\usepackage{textcomp} 

\newcommand{\cS}{\mathcal{S}}
\newcommand{\cP}{\mathcal{P}}
\newcommand{\cV}{\mathcal{V}}
\newcommand{\cN}{\mathcal{N}}
\newcommand{\cR}{\mathcal{R}}
\newcommand{\cT}{\mathcal{T}}
\newcommand{\cC}{\mathcal{C}}
\newcommand{\hP}{\hat{\cP}}
\newcommand{\hS}{\hat{\cS}}

\newcommand{\st}{\,\vert\,}
\newcommand{\floor}[1]{\left\lfloor{#1}\right\rfloor}
\newcommand{\ceil}[1]{\left\lceil{#1}\right\rceil}
\renewcommand{\o}[1]{\overline{#1}}

\newtheorem{theorem}{Theorem}[section]

\newtheorem{corollary}[theorem]{Corollary}
\newtheorem{lemma}[theorem]{Lemma}
\newtheorem{proposition}[theorem]{Proposition}

\theoremstyle{definition}

\newtheorem{definition}[theorem]{Definition}
\newtheorem{example}[theorem]{Example}

\title[Foldability of Words]{$k$-Foldability of Words}

\author[Bjorkman et al.]{Beth Bjorkman}
\address[Bjorkman]{Department of Mathematics, Iowa State University
\newline \indent \textnormal{Work completed while affiliated with Washington University in St. Louis.}}
\email{bjorkman@iastate.edu}

\author[]{Garner Cochran}
\address[Cochran]{Department of Mathematics, University of South Carolina}
\email{gcochran@math.sc.edu}

\author[]{Wei Gao}
\address[Gao]{Department of Mathematics and Statistics, Auburn University}
\email{wzg0021@auburn.edu}

\author[]{Lauren Keough}
\address[Keough]{Mathematics Department, Grand Valley State University}
\email{keoulaur@gvsu.edu}

\author[]{Rachel Kirsch}
\address[Kirsch]{Department of Mathematics, University of Nebraska - Lincoln}
\email{rkirsch@huskers.unl.edu}

\author[]{Mitch Phillipson}
\address[Phillipson]{School of Natural Sciences, St. Edward's University}
\email{mphilli2@stedwards.edu}

\author[]{Danny Rorabaugh}
\address[Rorabaugh]{Department of Mathematics and Statistics, Queen's University}
\email{rorabaugh@mast.queensu.ca}

\author[]{Heather Smith}
\address[Smith]{School of Mathematics, Georgia Institute of Technology}
\email{heather.smith@math.gatech.edu}

\author[]{Jennifer Wise}
\address[Wise]{\textnormal{Work completed while affiliated with the \textsc{University of Illinois at Urbana-Champaign.}}}
\email{jiwise2@illinois.edu}

\begin{document}


\begin{abstract}
We extend results regarding a combinatorial model introduced by Black, Drellich, and Tymoczko (2017+) which generalizes the folding of the RNA molecule in biology. 
Consider a word on alphabet $\{A_1, \o{A}_1, \ldots, A_m, \o{A}_m\}$ in which $\o{A}_i$ is called the complement of $A_i$. 
A word $w$ is foldable if can be wrapped around a rooted plane tree $T$, starting at the root and working counterclockwise such that one letter labels each half edge and the two letters labeling the same edge are complements. 
The tree $T$ is called $w$-valid. 

We define a bijection between edge-colored plane trees and words folded onto trees. 
This bijection is used to characterize and enumerate words for which there is only one valid tree. 
We follow up with a characterization of words for which there exist exactly two valid trees. 

In addition, we examine the set $\cR(n,m)$ consisting of all integers $k$ for which there exists a word of length $2n$ with exactly $k$ valid trees. 
Black, Drellich, and Tymoczko showed that for the $n$th Catalan number $C_n$, $\{C_n,C_{n-1}\}\subset \cR(n,1)$ but $k\not\in\cR(n,1)$ for $C_{n-1}<k<C_n$. 
We describe a superset of $\cR(n,1)$ in terms of the Catalan numbers by which we establish more missing intervals.  
We also prove $\cR(n,1)$ contains all non-negative integer less than $n+1$.
\end{abstract}


\maketitle

\section{Introduction}

The molecule ribonucleic acid (RNA) consists of a single strand of the four nucleotides adenine, uracil, cytosine, and guanine. In short, RNA is representable by finite sequences (or words) from the alphabet $A$, $U$, $C$, and $G$, lending itself to combinatorial study. In contrast to the double helix of DNA, the single-stranded nature of RNA often results in RNA folding onto itself as the nucleotides form bonds. As in DNA, we have the Watson-Crick pairs so that $C$ and $G$ form bonds and $A$ and $U$ form bonds. However, RNA has one more bond that may form, the wobble pair $G$ and $U$. 
It is worth noting that when RNA folds onto itself, not all nucleotides on a strand form bonds. Predicting the folded structure of RNA is important as the folded structure gives indication of its functionality. 

In this paper, we direct our attention to a generalized combinatorial model, motivated by the folding of RNA. This model was first introduced by Black, Drellich, and Tymoczko~\cite{black2015} with an initial restriction made to the Watson-Crick bonding
pairs, leaving the potential $GU$ bond for future study. With this restriction, we relabel our words to use the letters $A_1$, $\o{A}_1$, $A_2$, $\o{A}_2$ where $A_1$ only bonds with $\o{A}_1$ and $A_2$ only bonds with $\o{A}_2$. 

Further, we do not limit ourselves to an alphabet with only four letters. 
In particular, fix an integer $m\geq 1$ and expand the alphabet to $A_1, \o{A}_1, A_2, \o{A}_2, \ldots, A_m, \o{A}_m$ where $A_i$ and $\o{A}_i$ are called \emph{complements} and $A_i$ may only form a bond with $\o{A}_i$ and vice versa. We say that this is an alphabet on $m$ letters and their complements. Define the length of a word $w$ to be the number of letters in the word, letting $\varepsilon$ be the word of length zero (the empty word). 

As in~\cite{black2015}, we assume that when a word folds onto itself, every letter is matched with exactly one other letter. Thus, a \emph{folding} of a word can be represented by a non-crossing perfect matching of the letters so that two matched letters are complements.

Recall that the Catalan numbers enumerate the non-crossing perfect matchings on $2n$ points (see \cite{stanleyec2}, problem 6.19, part o). 
In our model, the underlying word restricts the allowable matching edges based on the letter corresponding to each point. 
However, the word $A_1 \o{A}_1A_1 \o{A}_1\ldots A_1 \o{A}_1$ admits every non-crossing perfect matching. 

Section~\ref{sec:prelim} contains some preliminaries and background from the work of Black, Drellich, and Tymoczko~\cite{black2015}. 
In Section~\ref{sec:1fold}, we define a bijection between foldings of words and edge-colored plane trees. 
This is used to enumerate the words which fold in precisely one way, a problem posed in \cite{black2015}. 
We also characterize $2$-foldable words by a decomposition in terms of $1$-foldable words. 
Section ~\ref{sec:R(n,m)} is devoted to studying the set $\cR$ of integers $k$ such that there is a word which folds in precisely $k$ ways. We give a superset of $\cR$, making a strong connection with Catalan numbers. 
In search of the smallest value which is not found in $\cR$, we also determine a large consecutive set of small values in $\cR$.

\section{Preliminaries}\label{sec:prelim}

In addition to non-crossing perfect matchings, the Catalan numbers also enumerate plane trees. A \emph{plane tree} is a straight line drawing of a rooted tree embedded in the plane with the root above all other vertices. This induces a left-to-right ordering of the children of a vertex. To obtain an ordering on the half edges of a plane tree, start at the root and trace the perimeter counterclockwise, touching each side of an edge exactly once. 

Fix a word $w=w[1]w[2]\cdots w[2n]$ and let $T$ be a plane tree with $n$ edges. Following the order of the half edges, label the $i^{th}$ half edge of $T$ with $w[i]$. We say that $T$ is \emph{$w$-valid} if for each edge of $T$, the two letters
from $w$ which label that edge are complements. 
\begin{definition}
 A word $w$ is said to be \emph{foldable} if there is a plane tree that is $w$-valid. 
 For integer $k\ge 0$, a word $w$ is \emph{$k$-foldable} if there are exactly $k$ plane trees that are
$w$-valid.
\end{definition}

For example, 
$w=A_1\o{A}_1A_1A_2\o{A}_2\o{A}_1$ is $2$-foldable as seen in Figure~\ref{fig-prelim-basic}. The corresponding non-crossing perfect matchings are also given. Further, the word $w_n=(A_1\o{A}_1)^n$ is $C_n$-foldable as every plane tree with $n$ edges is $w_n$-valid.

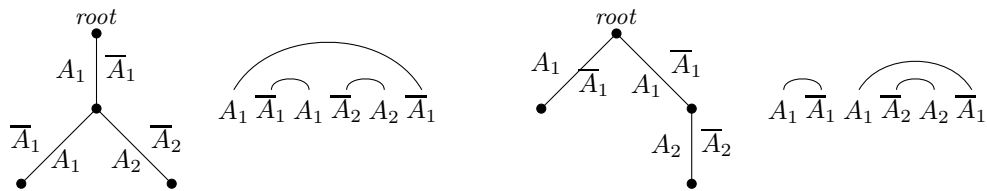
\begin{figure}[!ht]
 \begin{tikzpicture}
 \foreach \x in {1,2} \draw (0,\x) [fill=black] circle (0.06);
 \draw (-1,0)--(0,1)--(0,2);
 \draw [above] (0,2) node {\textit{\small{root}}};
 \draw (1,0) [fill=black] circle (0.06);
 \draw (-1,0) [fill=black] circle (0.06);
 \draw (0,1)--(1,0);
 \draw (0,1.5) [left] node {$A_1$};
 \draw (-.4,.3)  node {$A_1$};
 \draw (-.6,.6) [left] node {$\o{A}_1$};
 \draw (.4,.3)  node {$A_2$};
 \draw (.6,.6) [right] node {$\o{A}_2$};
 \draw (0,1.54) [right] node {$\o{A}_1$};
 \end{tikzpicture}
 \hspace{0in}
 \begin{tikzpicture}
 \foreach \x in {1,2} \draw (\x,-.05) node {$A_1$};
 \foreach \x in {3} \draw (\x,-.05) node {$A_2$};
 \foreach \x in {1,3} \draw (\x+.5,0) node {$\o{A}_1$};
 \foreach \x in {2} \draw (\x+.5,0) node {$\o{A}_2$};
 \draw (1,0.25) to [out=60, in=120] (3.5,0.25);
 \foreach \x in {1,2} \draw (\x+.5,0.25) to [out=90, in =90] (\x+1,0.25);
 \draw [color=white](2,-1) circle (0.06);
 \end{tikzpicture}
 \hspace{.3in}
 \begin{tikzpicture}
 \foreach \x in {1} \draw (0,\x) [fill=black] circle (0.06);
 \draw [above] (0,1) node {\textit{\small{root}}};
 \draw (-1,0)--(0,1);
 \draw (1,0) [fill=black] circle (0.06);
 \draw (-1,0) [fill=black] circle (0.06);
 \draw (0,1)--(1,0);
 \draw (-.3,.35)  node {$\o{A}_1$};
 \draw (-.6,.6) [left] node {${A}_1$};
 \draw (.4,.3)  node {$A_1$};
 \draw (.6,.6) [right] node {$\o{A}_1$};
 \draw (1,-1) [fill=black] circle (0.06);
 \draw (1,-1)--(1,0);
 \draw (1,-.5) [left] node {$A_2$};
 \draw (1,-.45) [right] node {$\o{A}_2$};
 \end{tikzpicture}
 \hspace{0in}
  \begin{tikzpicture}
 \foreach \x in {1,2} \draw (\x,-.05) node {$A_1$};
 \foreach \x in {3} \draw (\x,-.05) node {$A_2$};
 \foreach \x in {1,3} \draw (\x+.5,0) node {$\o{A}_1$};
 \foreach \x in {2} \draw (\x+.5,0) node {$\o{A}_2$};
 \draw (2,0.25) to [out=60, in=120] (3.5,0.25);
 \foreach \x in {1,2.5} \draw (\x,0.25) to [out=90, in =90] (\x+.5,0.25);
 \draw [color=white](2,-1) circle (0.06);
 \end{tikzpicture}
 \caption{The two foldings of $w=A_1\o{A}_1A_1A_2\o{A}_2\o{A}_1$ and their corresponding non-crossing perfect matchings.}
 \label{fig-prelim-basic}
\end{figure}

Black, Drellich, and Tymoczko~\cite{black2015} defined the following greedy algorithm to produce a folding of $w$. 
Given a word $w$ of length $n$, for each $i \leq n$ starting at $i=1$, create a matching as follows: 
Match $w[i]$ with $w[j]$ provided $j$ is the largest index such that $j<i$, $w[j]$ is not yet matched, and $w[j]$ is a complement of $w[i]$. 
If no such $j$ exists, leave $w[i]$ (temporarily) unmatched. 
If $w$ is foldable, this algorithm will produce a non-crossing perfect matching of $w$\cite{black2015}; 
the folding produced by the greedy algorithm is called the \emph{greedy folding}. 
The folding on the right in Figure~\ref{fig-prelim-basic} is the greedy folding. 

We will examine the following four sets in more detail.

\begin{definition}[Black, Drellich, and Tymoczko~\cite{black2015}]
Fix $n,m\in \mathbb{Z}^+$ and $k\in \mathbb{Z}^+ \cup \{0\}$. Let $\cS = \mathcal{S}(n,m)$ be the collection of words of length $2n$ from an alphabet 
with $m$ letters and their complements. For $w\in \cS$, define the 
following quantities:
\begin{itemize}
\item $\cP(n,m)$ is the set of words in $\cS$ that are foldable.
\item $\cS_k(n,m)$ is the set of words in $\cS$ that are $k$-foldable.\footnote{Note that~\cite{black2015} uses $\cN(n,m,k)$ instead of $\cS_k(n,m)$.}
\item $\cV(w)$ is the set of plane trees that are $w$-valid.
\item $\cR(n,m)$ is the set of integers $k$ for which $\cS_k(n,m)$ is non-empty.
\end{itemize}
\end{definition}

The set $\cS(n,m)$ can also be viewed as the length-$2n$ elements of the free group on $m$ generators. 
However, we are primarily interested in foldable words, which are precisely those that reduce to the identity element in the free group, so we make no further group theoretic connections. 

Let $w\in \mathcal{P}(n,m)$. 
Heitsch, Condon, and Hoos~\cite{heitsch} defined a local move to transform one plane tree in $\mathcal{V}(w)$ into another. 
For two trees in $\cV(w)$, there is a move from one to the other if there is a pair of edges that can be re-paired as in Figure~\ref{fig-prelim-decomp}. 
This defines a directed graph $G_w$ with a vertex for each plane tree in $\mathcal{V}(w)$ and an edge from $T_1$ to $T_2$ when there is a Type~1 move from $T_1$ to $T_2$. The following were proved in~\cite{black2015}.
\begin{theorem}[See Section~3 in~\cite{black2015}]\label{thm-graph}
Let $w$ be a foldable word.
\begin{enumerate}
 \item The greedy folding is a unique source of $G_w$.
 \item If $T_0$ is the greedy folding and $T\in \mathcal{V}(w)$,
  then there exists a path in $G_w$ from $T_0$ to $T$.
\end{enumerate}
\end{theorem}

\begin{figure}[!ht]
 \centering
 
 \begin{tikzpicture}
 \foreach \x in {(-2,0),(0,2), (2,0)} \draw \x [fill=black] circle (0.06);
 \draw (-2,0)--(0,2)--(2,0);
 \draw [dashed] (0,2)..controls+(-.4,.8)..(0,3)..controls+(.4,-.2)..(0,2);
 \draw [dashed] (0,2)..controls+(-.4,-.8)..(0,1)..controls+(.4,.2)..(0,2);
 \draw [dashed] (-2,0)..controls+(-.75,-.25)..(-2.75,-.75)..controls+(.5,0)..(-2,0);
 \draw [dashed] (2,0)..controls+(.75,-.25)..(2.75,-.75)..controls+(-.5,0)..(2,0);
 \draw (-2.4,-.4) node {2};
 \draw (2.4,-.4) node {4};
 \draw (0,2.65) node {1};
 \draw (0,1.35) node {3};
 \draw (-1.2,1.2) node  {$A$};
 \draw (-.85,.7) node  {$\o{A}$};
 \draw (1.2,1.2) node  {$\o{A}$};
 \draw (.9,.65) node  {${A}$};
 
 \begin{scope}[shift={+(.5,0)}]
 \draw [->,>=stealth, line width=1.5pt] (2.5,1.2)--(4.5,1.2);
  \draw [->,>=stealth,  line width=1.5pt] (4.5,.8)--(2.5,.8);
  \draw (3.5,1.2) node [above] {Type~1};
  \draw (3.5,.8) node [below] {Type~2};
  \end{scope}
  
  \begin{scope}[shift={+(1,0)}]
 \foreach \x in {-.5,1,2.5} \draw [fill=black] (6,\x) circle (0.06);
 \draw (6,-.5)--(6,2.5);
 \draw [dashed] (6,2.5)..controls+(-.4,.8)..(6,3.5)..controls+(.4,-.2)..(6,2.5);
 \draw [dashed] (6,-.5)..controls+(-.4,-.8)..(6,-1.5)..controls+(.4,.2)..(6,-.5);
 \draw [dashed] (6,1)..controls+(-.75,0)..(5,.5)..controls+(.5,-.1)..(6,1);
 \draw [dashed] (6,1)..controls+(.75,0)..(7,.5)..controls+(-.5,-.1)..(6,1);
 \draw (6,3.1) node {1};
 \draw (6,-1.1) node {3};
 \draw (6.6,.7) node {4};
 \draw (5.4,.7) node {2};
 \draw (6,1.79) node [right] {$\o{A}$};
 \draw (6,1.75) node [left] {${A}$};
 \draw (6,.15) node [right] {${A}$};
 \draw (6,.19) node [left] {$\o{A}$};
 \end{scope}
 \end{tikzpicture}
 
 \caption{Local moves between valid trees}
 \label{fig-prelim-decomp}
\end{figure}
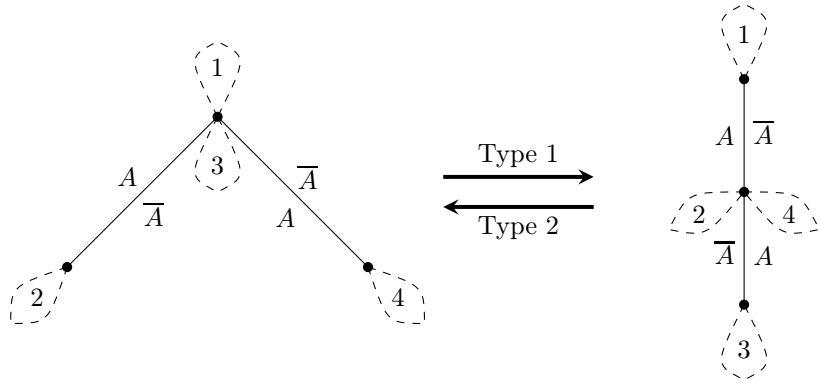

\section{Characterization and enumeration of foldable words}\label{sec:1fold}

In this section we give a bijection between foldings of words and edge-colored plane trees. 
Using this bijection we characterize both $1$-foldable and $2$-foldable words. An enumeration of $1$-foldable words is also given.

\subsection{Doubled Alphabet}

Fix a foldable word $w=w[1]w[2] \ldots w[2n]$. 
In any folding of $w$, if $w[i]$ bonds with $w[j]$, then $i$ and $j$ must have different parities because the subword $w[i+1] \cdots w[j-1]$ must be foldable and hence has an even (possibly zero) 
number of letters.
This leads to the notion of a \emph{doubled alphabet} to reflect that 
 $\{\text{odd } A_i,\text{ even }\o{A}_i\}$ and $\{\text{even }A_i,\text{ odd }\o{A}_i\}$ are the only possible bonds between an $A_i$ and an $\o{A}_i$. 
For an alphabet $\{A_1,\o{A}_1,\ldots,A_m,\o{A}_m\}$ and a word $w\in \mathcal{S}(n,m)$, define $\hat{w} \in \mathcal{S}(n,2m)$ on the doubled 
alphabet, $\{A_1,\o{A}_1,\ldots,A_{2m},\o{A}_{2m}\}$, 
as follows: 
\begin{itemize}
\item If $w[2\ell] = A_i$, then $\hat{w}[2\ell]=\o{A}_{m+i}$.
\item If $w[2\ell+1] = \o{A}_i$, then $\hat{w}[2\ell+1]=A_{m+i}$. 
\end{itemize} 

\begin{definition}
Fix $n,m\in \mathbb{Z}^+$.
\begin{itemize}
\item $\hS(n,m)$ is the set of words $w \in \cS(n,m)$ for which each letter in an odd-index position is from $\{A_1, A_2, \ldots, A_m\}$, and each letter in even-index position is from $\{\o{A}_1, \o{A}_2, \ldots, \o{A}_m\}$.
\item $\hP(n,m) \coloneqq \hS(n,m) \cap \cP(n,m)$.
\end{itemize}
\end{definition}

\begin{proposition}
For $n,m\in \mathbb{Z}^+$, the map $w \mapsto \hat{w}$ defines two bijections:
\[ \cS(n,m) \longleftrightarrow \hS(n,2m) \]
and
\[ \cP(n,m) \longleftrightarrow \hP(n,2m). \] 
\end{proposition}

\subsection{Walks on Regular Trees}

This alternation between letters from $\{A_1, A_2, \ldots, A_m\}$ and letters from $\{\o{A}_1, \o{A}_2, \ldots, \o{A}_m\}$ gives us greater ability to enumerate foldable words. 
To demonstrate, let us view words in $\hS$ as walks on an infinite regular tree. 
The infinite, unrooted, $m$-regular tree $T_m$ has $m$ distinct edges incident to every vertex, so for convenience, we can use the label set $\{A_1, A_2, \ldots, A_m\}$.
Given a walk on $T_m$, write down the sequence of edge labels, but on even-index steps, write down the complement of the labeling letter instead of the letter. 
So from any fixed vertex, walks of length $2n$ are in bijection with the elements of $\hS(n,m)$. 

A walk is \emph{closed} if it ends where it begins. 
Note that a walk is closed precisely when its corresponding word in $\cS$ is foldable. 
That is, closed walks on $T_m$ from a fixed vertex are in bijection with the elements of $\hP(n,m)$. 
These were enumerated by Quenell:
\begin{theorem}[Equation~(19) in \cite{quenell1994}] \label{thm-enum-closed-walks}
    Fix integer $m\geq 2$ and vertex $v$ of $T_m$. 
	The generating function $f_m(x)$ for the number $a_n$ of length-$2n$ closed walks on $T_m$ starting at $v$ is
\begin{eqnarray*}
    f_m(x) = \sum_{n=0}^{\infty} a_nx^n & = & \frac{2(m-1)}{m-2+m\sqrt{1 - 4(m-1)x}} \\
	& = & \frac{2-m\left(1-\sqrt{1 - 4(m-1)x}\right)}{2(1-m^2x)}.
\end{eqnarray*}
\end{theorem}


\begin{corollary}
	For integers $m\geq 2$ and $n \geq 1$ and vertex $v$ of $T_m$, the number $a_n$ of length-$2n$ closed walks on $T_m$ starting at $v$ is
\begin{eqnarray*}
	a_n & = & m^{2n} - \sum_{i=1}^n \frac{m^{(1 + 2(n-i))}(m-1)^{i}}{4i-2}\binom{2i}{i}. 
\end{eqnarray*}
\end{corollary}

\begin{proof}

By Newton's generalized binomial theorem, 
\begin{eqnarray*}
	\sqrt{1 - 4(m-1)x} & = & \sum_{n=0}^\infty \binom{1/2}{n}(-4(m-1)x)^n  \\
	& = & \sum_{n=0}^\infty \frac{ \prod_{i = 0}^{n-1} \left(\frac{1}{2} - i \right) }{n!} (-4(m-1)x)^n \\
	 & = & \sum_{n=0}^\infty \frac{-(m-1)^{n}}{2n-1}\binom{2n}{n} x^n.
\end{eqnarray*}

Substituting this back into $f_m(x)$, we get
\begin{eqnarray*}
	f_m(x) & = & \frac{2-m\left(\sum_{n=1}^\infty \frac{(m-1)^{n}}{2n-1}\binom{2n}{n} x^n \right)}{2(1-m^2x)}. 
\end{eqnarray*}

Setting $b_i = \frac{m(m-1)^{i}}{4i-2}\binom{2i}{i}$ for $i \geq 1$, we have
\begin{eqnarray*}
	f_m(x) & = & \frac{1 - \sum_{n=1}^\infty b_n x^n}{1 - m^2x} \\
	& = & 1 + (m^2 - b_1)x + (m^2(m^2 - b_1) - b_2)x^2 +  \cdots \\
	& = & \sum_{n=0}^\infty \left(m^{2n} - \sum_{i=1}^n b_im^{2(n-i)}\right)x^n.
\end{eqnarray*}
\end{proof}

We can obtain asymptotics for $a_n$ using the Maple\texttrademark{} package \textbf{algolib} (version 17.0), or the saddle point method, on the generating function. 

\begin{corollary} \label{cor-closed-walk-asymp}
	For fixed $m\geq 3$ and vertex $v$ of $T_m$, the number $a_n$ of length-$2n$ closed walks on $T_m$ starting at $v$ is asymptotically
	\[
	a_n = \frac{(4m-4)^n}{n^{3/2}} \left(\frac{m(m-1)}{\sqrt{\pi}(m-2)^2}+ O\!\left(\frac{1}{\sqrt{n}}\right) \right) . 
	\]
\end{corollary}

	Recall that $\cP(n,m)$ is in bijection with $\hP(n,2m)$, which can be enumerated by closed walks on $T_{2m}$.
	Thus, using Corollary~\ref{cor-closed-walk-asymp} with $2m$, we have that for fixed $m\geq 2$ the number of foldable words of length $2n$ as $n$ approaches infinity is asymptotically
	\begin{equation} \label{eqn-foldable-asymp}
		|\cP(n,m)| = \Theta(n^{-3/2}(8m-4)^n) .
	\end{equation}

\subsection{Labeling Plane Trees}

There is a natural bijection between foldings of words and (not necessarily proper) edge-colorings of rooted plane trees which is most clearly seen by examining the foldings of $\hat{w}$ rather than $w$. 
More generally, we consider words in $\hP$---that is, with alternating unbarred and barred letters---rather than in $\cP$. 
Set 
\[\cT(n,m) \coloneqq \{(w,T)\st w\in\hP(n,m),\;T\in \cV(w)\},\]
where the element $(w,T)$ is viewed as a folding of $w$ around $T$.
With $[m]$ denoting the set $\{1,2,\ldots,m-1,m\}$, define 
\[\cC(n,m) \coloneqq \{(c,T)\st T\text{ is a plane tree with }n\text{ edges and } c:E(T) \to [m]\},\]
so that the elements of $\mathcal{C}(n,m)$ represent edge-colored plane trees, where the coloring is not necessarily proper. 

\begin{theorem} \label{thm-Tnm-Cnm}
    For all integers $n\geq 0$ and $m\geq 1$, $|\cT(n,m)| = |\cC(n,m)|$.
\end{theorem}

\begin{proof}
We will define a bijection from $\cT(n,m)$ to $\cC(n,m)$. 
Fix an arbitrary $(w,T)\in \cT(n,m)$. 
Define the edge-coloring $c:E(T) \to [m]$ so that $c(e)=i$ if the half edges of $e\in E(T)$ are labeled $(A_i,\o{A}_i)$ or $(\o{A}_i,A_i)$. Figure~\ref{fig-fold-color-ex} gives an example of this mapping.

\begin{figure}[!ht]
 \centering
 
 \begin{tikzpicture}[scale=1.2]
  \foreach \x in {(0,3),(-1,2),(1,2), (-2,1),(0,1),(-1,0),(1,0)} \draw \x [fill=black] circle (0.06);
  \draw (0,3) node[above] {\textit{\small{root}}};
  \draw (1,0)--(-1,2)--(0,3)--(1,2) (-2,1)--(-1,2) (-1,0)--(0,1);
  \draw (-.25,2.3) node {$\o{A}_1$};
  \draw (-.75,2.6) node {$A_1$};
  \draw (-1.25,1.3) node {${A}_3$};
  \draw (-1.75,1.6) node {$\o{A}_3$};
  \draw (-.25,.3) node {$\o{A}_1$};
  \draw (-.75,.6) node {$A_1$};
  \draw (.4,2.3) node {${A}_2$};
  \draw (.75,2.6) node {$\o{A}_2$};
  \draw (-.6,1.3) node {$\o{A}_3$};
  \draw (-.25,1.6) node {${A}_3$};
  \draw (.4,.3) node {${A}_3$};
  \draw (.75,.6) node {$\o{A}_3$};
 \end{tikzpicture}
 \hspace{.5in}
 \begin{tikzpicture}[scale=1.2]
  \foreach \x in {(0,3),(-1,2),(1,2), (-2,1),(0,1),(-1,0),(1,0)} \draw \x [fill=black] circle (0.06);
  \draw (0,3) node[above] {\textit{\small{root}}};
  \draw (1,0)--(-1,2)--(0,3)--(1,2) (-2,1)--(-1,2) (-1,0)--(0,1);
  \draw (-.65,2.6) node {$1$};
  \draw (-1.65,1.6) node {$3$};
  \draw (-.65,.6) node {$1$};
  \draw (.65,2.6) node {$2$};
  \draw (-.35,1.6) node {$3$};
  \draw (.65,.6) node {$3$};
 \end{tikzpicture}

 \caption{A folding of $d(w) = A_1\o{A}_3A_3\o{A}_3A_1\o{A}_1A_3\o{A}_3A_3\o{A}_1A_2\o{A}_2$ and the corresponding edge-coloring of the tree.}
 \label{fig-fold-color-ex}
\end{figure}
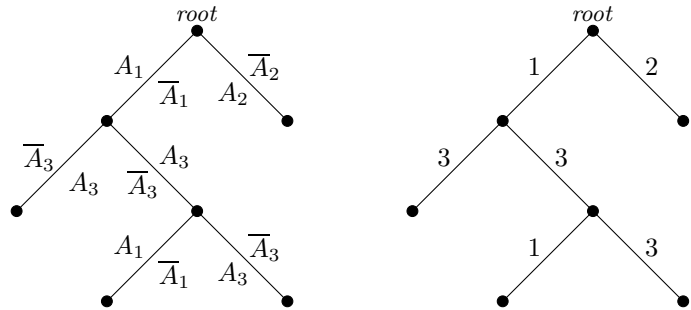

The inverse function is defined as follows.
Fix $(c,T)\in \mathcal{C}(n,m)$. 
The color of each edge indicates the two letters that will be assigned to its half edges. 
It only remains to determine which letter will be assigned to which half edge. 
In the ordering of the half edges of $T$, each edge will have an even half edge and an odd half edge. 
This is because the subtree below the edge contains an even number of half edges. 
Assign the letter with the bar to the even half edge and the other to the odd half edge. This labeling is precisely the folding of $w$ on $T$. Since an inverse exists, the function defined is an injection.
\end{proof}

\subsection{1-foldable classification and enumeration}

The correspondence between foldings of words and edge-colored plane trees leads to an enumeration of words which are $1$-foldable.  
Denote with $\cT_k(n,m)$ the set of all $(w,T) \in \cT(n,m)$ where $w$ is $k$-foldable. 
Then $|\cT_k(n,2m)| = k \cdot |\cS_k(n,m)|$ for any $k \in \mathbb{Z}^+$ and, in particular, $|\cT_1(n,2m)| = |\cS_1(n,m)|$. 

\begin{theorem} 
 The words in $\cS_1(n,m)$ are in bijection with the 
proper $2m$-edge-colorings of  plane trees with $n$ edges.
\label{thm-1fold-bij}
\end{theorem}

\begin{proof}

First recall that the graph of valid trees for a given word is connected (Theorem~\ref{thm-graph}). 
The bijection in Theorem~\ref{thm-Tnm-Cnm} can be used to detect available local moves from the edge-coloring. 
In particular, a move exists precisely when two incident edges have the same color. 
Therefore, elements of $\cC(n,m)$ with a proper edge-coloring correspond exactly with elements of $\cT_1(n,m)$. 
Hence, elements of $\cC(n,2m)$ with a proper edge-coloring are in correspondence with elements of $\cS_1(n,m)$. 
\end{proof}

Using this classification, we now enumerate $1$-foldable words. 

\begin{lemma} \label{lma-fold-colors}
 Let $T$ be a plane tree with degree multiset $\{1^{\alpha_1},2^{\alpha_2},\ldots\}$ where there are $\alpha_i$ vertices with degree $i$. 
 Then the number of proper $k$-edge-colorings of $T$ is 
  \begin{equation} \label{eqn-fold-colorings}
   k\prod_{i=1}^{\infty}\left( \binom{k-1}{i-1}(i-1)!\right)^{\alpha_i},
  \end{equation}
  where $0^0$ is understood to equal 1.
  \end{lemma}

\begin{proof}
Fix a leaf $v$. Let $u$ be the unique neighbor of $v$ and let $\deg(u)$ be the degree of $u$. 
Color the edge $\{v,u\}$ with one of $k$ colors. 
Next, color the remaining $\deg(u)-1$ incident edges, which can be ordered in $(\deg(u)-1)!$ ways. 
Visiting the vertices via a breadth-first search, each vertex will contribute a similar factor to the product as all but one of its incident edges will already be colored. 
\end{proof}

Let $\Delta(T)$ be the maximum degree in $T$. 
Note that if $\Delta(T)>k$, expression~\eqref{eqn-fold-colorings} collapses to zero as expected since $\Delta(T)$ colors are required for a proper coloring of the edges incident to a maximum-degree vertex. 

\begin{lemma}[Mallows and Wacher~\cite{mallowswacher}] \label{lma-fold-treenum}
Let $RPT(\alpha_1,\alpha_2,\ldots)=RPT(\mathbf{\alpha})$ be the number of 
 plane trees with degree multiset $\{1^{\alpha_1},2^{\alpha_2},\ldots\}$. 
Then
 \begin{equation}
  RPT(\mathbf{\alpha}) = \frac{2}{\alpha_1}\binom{1+\alpha_2+2\alpha_3+3\alpha_4+\cdots}{\alpha_1-1,\alpha_2,\alpha_3,\ldots}.
 \end{equation}
\end{lemma}

For any sequence of non-negative integers $(\alpha_2,\alpha_3,\ldots)$, there is a plane tree with degree multiset $\{1^{\alpha_1},2^{\alpha_2},\ldots\}$ for an appropriate choice of $\alpha_1$, the number of leaves. 
In particular, $RPT(\alpha_1, \alpha_2, \ldots)\neq 0$ if and only if 
\begin{equation} \label{eqn-fold-a1}
 \alpha_1 = 2 + \alpha_3 + 2\alpha_4 + 3\alpha_5 + 4\alpha_6 + \cdots = 2 + \sum_{i=2}^\infty (i-2)\alpha_i. 
\end{equation}

\begin{lemma} \label{lma-deg-conds}
 For the multiset $\{2^{\alpha_2},3^{\alpha_3},\ldots\}$, 
  set $\alpha_1$ as in~\eqref{eqn-fold-a1}, 
  and let $T$ be a plane tree with the degree multiset $\{1^{\alpha_1}, 2^{\alpha_2},3^{\alpha_3},\ldots\}$. 
 Then $\alpha_1,\alpha_2,\ldots$ satisfy the conditions
 \begin{enumerate}
  \item $\alpha_i=0$ for $i> 2m$, and
  \item $\displaystyle\sum_{i=1}^{2m} i\alpha_i = 2n$,
 \end{enumerate}
  if and only if $T$ has $n$ edges and can be properly edge-colored with $2m$ colors. 
\end{lemma}

\begin{proof}
 For any tree $T$, the edge chromatic number $\chi'(T) = \Delta(T)$. 
 Condition~1 is equivalent to saying $\Delta(T) \le 2m$, so $T$ is $2m$-edge-colorable. 
 Condition~2 is equivalent to $T$ having $n$ edges. 
\end{proof}
Lemma~\ref{lma-deg-conds} gives an explicit characterization of the degree conditions for a plane tree to have a proper edge-coloring. 
Having previously established a correspondence between foldings of $1$-foldable words and proper edge-colorings of plane trees, the following theorem is now clear. 
\begin{theorem} \label{thm-1-fold}
 The number of $1$-foldable words of length $2n$ on an alphabet with $m$ letters and their complements is
 \begin{equation} \label{eqn-1-fold}
  \sum \frac{2}{\alpha_1}\binom{n}{\alpha_1-1,\alpha_2,\alpha_3,\ldots,\alpha_{2m}}\cdot 
2m\prod_{i=1}^{2m}\left( \binom{2m-1}{i-1}(i-1)!\right)^{\alpha_i},
 \end{equation}
  where the sum is over all non-negative sequences $(\alpha_1, \alpha_2,\alpha_3,\ldots,\alpha_{2m})$ 
  such that $\sum_{i=1}^{2m} i\alpha_i = 2n$ and $\alpha_1 = 2 + \sum_{i=2}^{2m}(i-2)\alpha_i$.
\end{theorem}

\begin{proof}
 By Theorem~\ref{thm-1fold-bij} and Lemmas~\ref{lma-fold-colors}, \ref{lma-fold-treenum}, and~\ref{lma-deg-conds} with the observation that
\[ n = 1 + \alpha_2 + 2 \alpha_3 + \cdots + (2m-1)\alpha_{2m}.\]
\end{proof}

\begin{example}
When $m=1$, expression~\eqref{eqn-1-fold} is a summation with one term, when $(\alpha_1, \alpha_2)=(2, n-1)$, and gives $2n$ words of length $2n$ which are $1$-foldable. 
The $2n$ words are exactly $\{A^i \o{A}\,^n A^{n-i}: i\in[n]\}$ and $\{\o{A}\,^i A^n\o{A}\,^{n-i}: i\in [n]\}$.
\end{example}

\begin{example}
When $m=2$, each term in~\eqref{eqn-1-fold} has the following form:
\begin{align*}
    & \frac{2}{\alpha_1}\binom{n}{\alpha_1-1,\alpha_2,\alpha_3,\alpha_4}\cdot 
4\prod_{i=1}^{4}\left( \binom{3}{i-1}(i-1)!\right)^{\alpha_i} \\
    &= \frac{8}{\alpha_1}\cdot\binom{n}{\alpha_1-1,\alpha_2,\alpha_3,\alpha_4}\cdot(1\cdot 0!)^{\alpha_1}\cdot(3\cdot 1!)^{\alpha_2} \cdot (3\cdot2!)^{\alpha_3} \cdot (1\cdot 3!)^{\alpha_4} \\
    &= \frac{8}{3(2+\alpha_3+2\alpha_4)}\cdot\binom{n}{1+\alpha_3+2\alpha_4,n-2\alpha_3-3\alpha_4-1,\alpha_3,\alpha_4}\cdot 3^{n - \alpha_3 - 2\alpha_4} \cdot 2^{\alpha_3 + \alpha_4},
\end{align*}
with $\alpha_3$ and $\alpha_4$ positive integers such that $n > 2\alpha_3 + 3\alpha_4$. 
The multinomial coefficient pulls the maximum of this term away from the boundaries; that is, as $n$ grows, the maximum is not found where one of $\{\alpha_1-1,\alpha_2,\alpha_3,\alpha_4\}$ is $o(n)$. 
Let us therefore assume that $\lim_{n \to \infty} \frac{\alpha_3}{n} = x$ and $\lim_{n \to \infty} \frac{\alpha_4}{n} = y$ for positive constants $x$ and $y$ with $2x + 3y < 1$.
Applying Stirling's approximation, this term is asymptotically
\begin{align*}
    & \frac{(1 + o(1))^n}{ (x+2y)^{n(x+2y)} \cdot (1 - 2x - 3y)^{n(1 - 2x - 3y)} \cdot x^{xn} \cdot y^{yn} } \cdot 3^{n(1 - x - 2y)} \cdot 2^{n(x + y)} \\
    &= \left[ \frac{3 + o(1)}{1-2x-3y} \cdot \left(\frac{2(1-2x-3y)^2}{3x(x+2y)}\right)^{\!x} \cdot \left(\frac{2(1-2x-3y)^3}{9y(x+2y)^2}\right)^{\!y} \right]^n.
\end{align*}
Numerically, the base of this exponential is maximized when $(x,y) \approx (0.22103,0.07050)$, 
 which gives a maximum term of $(8.65936223 \pm 2^{-25})^n$. 
Since the number of terms in the sum is polynomial in $n$ for fixed $m$, this is also an asymptotic approximation for the whole sum. 
Compare this to the $16^n$ length-$2n$ words on an alphabet of $m=2$ letters and their complements, of which $(12+o(1))^n$ are foldable by equation~\eqref{eqn-foldable-asymp}. 
\end{example}

\subsection{2-foldable classification}

Using the bijection with edge-colored trees, we can also classify $2$-foldable words. 
In particular, a foldable word is $2$-foldable if the edge-colored tree corresponding to the greedy folding has only one pair of incident edges with the same color, 
 and the tree that results after making the corresponding Type~1 move at those edges has only one pair of incident edges with the same color. 

An equivalent characterization is those trees that have an $A$-decomposition, defined as follows. 

\begin{definition} An $A$-decomposition of a word $w$ is a list of words $u_1$, $u_2$, $u_3$, $v_1$, and $v_2$ such that

\begin{enumerate}
		\item $w = u_1Av_1\o{A}u_2Av_2\o{A}u_3$ for some (possibly barred) letter $A$,
		\item the words $u_1u_3$, $u_2$, $v_1$, and $v_2$ are foldable, and
		\item the words $u_1u_2A\o{A}u_3$ and $Av_1v_2\o{A}$ are $1$-foldable.
\end{enumerate}	
\end{definition}

Note that we consider here any word in $\cS$, not necessarily with an alternating bar pattern. 
Moreover, $A$ in condition~(1) may be a barred letter. 
In this case, $\o{A}$ signifies its unbarred complement. 

\begin{figure}[!ht]
 \centering
 
 \begin{tikzpicture}
 \foreach \x in {(-2,0),(0,2), (2,0)} \draw \x [fill=black] circle (0.06);
 \draw (-2,0)--(0,2)--(2,0);
\draw [dashed] (0,2)..controls+(-.5,.8)..(0,3)..controls+(.5,-.2)..(0,2);
 \draw [dashed] (0,2)..controls+(-.4,-.8)..(0,1)..controls+(.4,.2)..(0,2);
 \draw [dashed] (-2,0)..controls+(-.75,-.25)..(-2.75,-.75)..controls+(.5,0)..(-2,0);
 \draw [dashed] (2,0)..controls+(.75,-.25)..(2.75,-.75)..controls+(-.5,0)..(2,0);
 \draw (-2.4,-.4) node {$v_1$};
 \draw (2.4,-.4) node {$v_2$};
 \draw (0,2.65) node {$u_1u_3$};
 \draw (0,1.35) node {$u_2$};
 \draw (-1.2,1.2) node  {$A$};
 \draw (-.85,.7) node  {$\o{A}$};
 \draw (1.2,1.2) node  {$\o{A}$};
 \draw (.9,.65) node  {$A$};

  \begin{scope}[shift={+(1,0)}]
 \foreach \x in {-.5,1,2.5} \draw [fill=black] (6,\x) circle (0.06);
 \draw (6,-.5)--(6,2.5);
 \draw [dashed] (6,2.5)..controls+(-.5,.8)..(6,3.5)..controls+(.5,-.2)..(6,2.5);
 \draw [dashed] (6,-.5)..controls+(-.4,-.8)..(6,-1.5)..controls+(.4,.2)..(6,-.5);
 \draw [dashed] (6,1)..controls+(-.75,0)..(5,.5)..controls+(.5,-.1)..(6,1);
 \draw [dashed] (6,1)..controls+(.75,0)..(7,.5)..controls+(-.5,-.1)..(6,1);
 \draw (6,3.1) node {$u_1u_3$};
 \draw (6,-1.1) node {$u_2$};
 \draw (6.6,.7) node {$v_2$};
 \draw (5.4,.7) node {$v_1$};
 \draw (6,1.79) node [right] {$\o{A}$};
 \draw (6,1.75) node [left] {$A$};
 \draw (6,.15) node [right] {$A$};
 \draw (6,.19) node [left] {$\o{A}$};
 \end{scope}
 \end{tikzpicture}
 
 \caption{Two valid trees of a word with an $A$-decomposition.}
 \label{fig-1-decomp}
\end{figure}
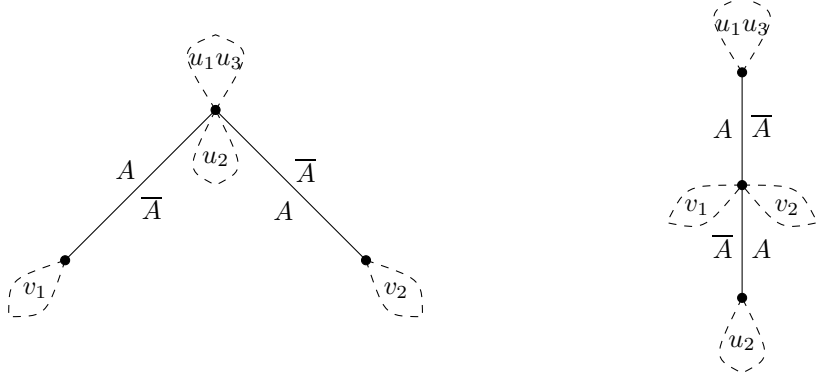

\begin{theorem} 
A word $w$ is $2$-foldable if and only if it has an $A$-decomposition.
\end{theorem}


\begin{proof}
Suppose $w$ has an $A$-decomposition. 
Then $w = u_1Av_1\o{A}u_2Av_2\o{A}u_3$ can be folded into the two edge-colored trees using parts~(1) and~(2) of the definition of $A$-decomposition. 
Parts~(2) and~(3) of the definition imply that $u_1u_3$, $v_1$, $u_2$, and $v_2$ are $1$-foldable. 
By part~(3) of the definition, the only incident edges with local moves 
in either of these two trees are the edges labeled by $A\o{A}$ or $\o{A}A$ in Figure~\ref{fig-1-decomp}. 
Thus there are no other local moves, and since the state space graph is connected, these are the only two foldings of $w$.

Now suppose $w$ is $2$-foldable. 
The two foldings correspond bijectively to two edge-colored trees, which are adjacent by local moves shown in Figure~\ref{fig-prelim-decomp}. 
Thus we have properties~(1) and~(2) of an $A$-decomposition, and~(3) follows from the fact that $w$ is $2$-foldable so has no other local moves.
\end{proof}

\section{The values in $\cR(n,m)$} \label{sec:R(n,m)}

In this section we develop a better understanding of the set $\cR(n,m)$ of all $k$ for which there is a word in $\cS(n,m)$ which is $k$-foldable. 
Black, Drellich, and Tymoczko~\cite{black2015} initiated this study with the following proposition. 

\begin{proposition}[\cite{black2015}]
Let $C_i \coloneqq \frac{1}{i+1} \binom{2i}{i}$, 
 the $i\textsuperscript{th}$ Catalan number.
For integers $n>1$ and $m>0$, 
 $\{C_{n-1},C_n\} \subset \cR(n,m)$ but if $C_{n-1}<k<C_n$ then $k\not\in \cR(n,m)$.
\label{prop-top-gap}
\end{proposition}



Wagner~\cite{wagner2015} further investigated $\cR(n,m)$ and established monotonicity in the following sense.

\begin{proposition}[\cite{wagner2015}]\label{prop-Rnm-n-monotone}
For positive integers $n$ and $m$, 
 $\cR(n,m) \subsetneq \cR(n+1,m).$
\end{proposition}
Note that the previous two propositions establish that $C_i\in \cR(n,m)$ for $1\leq i\leq n$. 
Wagner also showed monotonicity of $\cR(n,m)$ in $m$. 

\begin{proposition}[\cite{wagner2015}]
For positive integers $n$ and $m$, 
 $\cR(n,m) \subseteq \cR(n,m+1)$. 
Further, there exist $n$ and $m$ such that $\cR(n,m) \neq \cR(n,m+1)$. 
\end{proposition}

We focus our attention mainly on the case $m=1$. 
To give some indication of the values in $\cR(n,1)$, 
 computationally we find:

\begin{align*}
    \cR(0,1) = \{1\}; \hspace{.4in} \\
    \cR(1,1) = \cR(0,1) \cup \{& 0\}; \\
    \cR(2,1) =  \cR(1,1) \cup \{& 2\}; \\
    \cR(3,1) =  \cR(2,1) \cup \{& 5\}; \\
    \cR(4,1) =  \cR(3,1) \cup \{& 3,4,14\}; \\
    \cR(5,1) =  \cR(4,1) \cup \{& 7,10,42\}; \\
    \cR(6,1) =  \cR(5,1) \cup \{& 6,8,12,16,18,19,25,28,132\}; \\
    \cR(7,1) =  \cR(6,1) \cup \{& 9,15,20,30,40,43,52,56,70,84,429 \};\\
    \cR(8,1) =  \cR(7,1)\cup \{& 22,23,24,26,32,35,36,38,50,55,73,74,80,85,96, \\
    & 106,114,115,126,157,160,174,196,210,264,1430 \}.
\end{align*}

Working toward a more thorough understanding of the set $\cR(n,1)$,  
 we first construct a superset of $\cR(n,1)$ in Theorem~\ref{thm-Rnm-superset} providing some structure for the values that can appear in $\cR(n,1)$. 
From there, we determine intervals of integers which do not lie in the set, 
 such as integers in the interval $[C_{n-1}+1, C_n -1]$ from Proposition~\ref{prop-top-gap}. 
Then, in search of the smallest value $k$ such that $k\not\in \cR(n,1)$, 
 we conclude by proving $\{0,1,2,\ldots, n\}\subseteq \cR(n,1)$ and hence in $\cR(n,m)$.

\subsection{Catalan numbers and $\cR(n,1)$}

The Catalan numbers, which enumerate the  plane trees, are an integral part of the set $\cR(n,m)$. As already noted,  $C_t\in \cR(n,1)$ for all $1\leq t\leq n$. In fact, $C_t$ is the number of foldings of $(A\o{A})^t$, and this is the maximum number of foldings for a word of length $2t$. 
Theorem~\ref{thm-Rnm-superset} establishes a superset of $\cR(n,1)$ which highlights the fundamental nature of the Catalan numbers in the values of $\cR(n,1)$. 
The following discussion and examples motivate the theorem. 

As previously mentioned, for calculating the values in $\cR(n,1)$, it suffices to consider words in $\hP(n,2)$, foldable words which strictly alternate between unbarred and barred letters. 
For readability in this case, we will use $A$ and $B$ instead of $A_1$ and $A_2$. 
Fix such a foldable word $w$ with $t$ entries that are $A$ and $n-t$ entries that are $B$. 
Without loss of generality, assume $w$ begins with $A$. 
For example, let $w=A\o{A}(B\o{B})^5 A\o{A} (B\o{B})^7 A\o{A} (B\o{B})^4$. Here $t = 3$ and $n - t = 5 + 7 + 4 = 16$.

Now consider the maximal subwords (consecutive letters) of $w$ which consist of only the letters $B$ and $\o{B}$. 
We call these subwords \emph{maximal $B$-subwords.}
Let $\ell_1, \ell_2, \ldots, \ell_m$ be the number of letters in each of these maximal $B$-subwords. 
Consequently $m \leq 2t$ and $\sum_{i=1}^{m} \ell_i = 2(n-t)$. 
In our present example, $w$ has $m = 3$ maximal $B$-subwords with lengths $\ell_1=10$, $\ell_2=14$, and $\ell_3=8$.

Fix a non-crossing matching on the letters $A$ and $\o{A}$ in $w$. 
We will use the term \emph{$A$-matching} to refer to such a partial matching of $w$. 
Because of the alternating pattern of barred letters in the doubled alphabet any $A$-matching partitions the maximal $B$-subwords into groups of the form $(B\o{B})^s$ or $(\o{B}B)^s$ when concatenated, 
 where $2s$ is the sum of the corresponding $\ell_i$ values. 
We will refer to these as \emph{$B$-groupings}. 
Thus, for each $A$-matching $\varphi$, 
 there is at least one non-crossing perfect matching of $w$ which extends $\varphi$ since $w$ was foldable.

See Figure~\ref{fig-A-match-ex-1} for one possible $A$-matching on our example word $w$. 
With this $A$-matching, the $B$-subwords have been partitioned into $(B\o{B})^5$ with $2\cdot 5 = \ell_1$ and $(B\o{B})^{11}$ with $2\cdot 11 = \ell_2 + \ell_3$. 

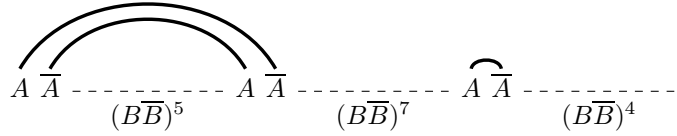
\begin{figure}[ht!]

 \begin{tikzpicture}
  \draw [above] (0,0) node {$A$};
  \draw [above] (.4,0) node {$\o{A}$};
  \draw [dashed](.7,.2)--(2.7,.2);
  \draw [below] (1.7,.2) node {$(B \o{B})^5$};
  \draw [above] (3,0) node {$A$};
  \draw [above] (3.4,0) node {$\o{A}$};
  \draw [dashed] (3.7,.2)--(5.7,.2);
  \draw [below] (4.7,.2) node {$(B \o{B})^7$};
  \draw [above] (6,0) node {$A$};
  \draw [above] (6.4,0) node {$\o{A}$};
  \draw [dashed] (6.7,.2)--(8.7,.2);
  \draw [below] (7.7,.2) node {$(B \o{B})^4$};
  \draw [line width=1.3] (0,.5) to [out=60, in=120] (3.4,.5);
  \draw [line width=1.3] (.4,.5) to [out=60, in=120] (3,.5);
  \draw [line width=1.3] (6,.5) to [out=80, in=100] (6.4,.5);
 \end{tikzpicture}

 \caption{An $A$-matching of $A\o{A}(B\o{B})^5 A\o{A} (B\o{B})^7 A\o{A} (B\o{B})^4$ 
which can be extended in $C_5 C_{11}$ ways.}
 \label{fig-A-match-ex-1}
\end{figure}

For a $B$-grouping of length $2t$, 
 there are $C_t$ ways to fold the group. 
Thus, each $A$-matching $\varphi$ of $w$ extends to $\prod_{i=1}^j C_{s_i}$ non-crossing perfect matchings of $w$, where $j$ is the number of $B$-groupings and each $s_i$ is half of the sum of the corresponding subset of $\{\ell_1,\ldots, \ell_m\}$. 
For the present example, the $A$-matching in Figure~\ref{fig-A-match-ex-1} extends to $C_5 C_{11}$ non-crossing matchings of $w$. 
Alternatively, for the $A$-matching $\varphi' = \{w[1]w[2], w[13]w[14], w[29]w[30]\}$, 
 there is nothing separating the maximal $B$-subwords, 
 so $\varphi'$ extends in $C_{5+7+4} = C_{16}$ ways. 

The following example highlights the structure that results from an $A$-matching when there are maximal $B$-subwords of odd length. 

\begin{example} \label{ex-superset2}
Let $w = A(\o{B} B)^4\o{B}A\o{A}(B\o{B})^6B\o{A}(B\o{B})^4 A \o{A}$.  
Here $\ell_1 = 9$, $\ell_2 = 13$ and $\ell_3 = 8$. 
Again, there are multiple non-crossing $A$-matchings, 
 but any such matching ensures that the maximal $B$-subwords of lengths $9$ and $13$ will be free to match with each other because they are the only two of odd length. 
One non-crossing $A$-matching is $\varphi = \{w[1]w[26],w[11]w[12], w[35]w[36]\}$ and another is $\varphi' = \{w[1]w[36], w[11]w[12],w[26]w[35]\}$. 
Both $\varphi$ and $\varphi'$ extend in $C_{11}\cdot C_4$ ways. 
(See Figure~\ref{fig-A-match-ex-2}.)
\end{example}

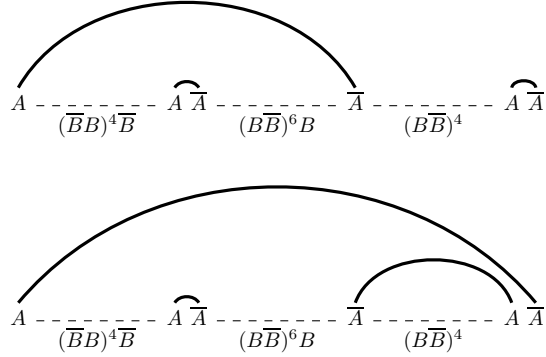
\begin{figure}[ht!]

 \begin{tikzpicture}[scale=.8, every node/.style={scale=.8}]
  \draw [above] (.4,0) node {${A}$};
  \draw [dashed] (.7,.2)--(2.7,.2);
  \draw [below] (1.7,.2) node {$(\o{B} B)^4 \o{B}$};
  \draw [above] (3,0) node {$A$};
  \draw [above] (3.4,0) node {$\o{A}$};
  \draw [dashed](3.7,.2)--(5.7,.2);
  \draw [below] (4.7,.2) node {$(B \o{B})^6 B$};
  \draw [above] (6,0) node {$\o{A}$};
  \draw [dashed] (6.3,.2)--(8.3,.2);
  \draw [below] (7.3,.2) node {$(B \o{B})^4$};
  \draw [line width=1.2] (.4,.5) to [out=60, in=120] (6,.5);
  \draw [line width=1.2] (3,.5) to [out=60, in=120] (3.4,.5);
  \draw [line width=1.2] (8.6,.5) to [out=80, in=100] (9,.5);
  \draw [above] (8.6,0) node {$A$};
  \draw [above] (9,0) node {$\o{A}$};
 \end{tikzpicture}
 \hspace{.3in}
 \begin{tikzpicture}[scale=.8, every node/.style={scale=.8}]
  \draw [above] (.4,0) node {${A}$};
  \draw [dashed] (.7,.2)--(2.7,.2);
  \draw [below] (1.7,.2) node {$(\o{B} B)^4 \o{B}$};
  \draw [above] (3,0) node {$A$};
  \draw [above] (3.4,0) node {$\o{A}$};
  \draw [dashed](3.7,.2)--(5.7,.2);
  \draw [below] (4.7,.2) node {$(B \o{B})^6 B$};
  \draw [above] (6,0) node {$\o{A}$};
  \draw [dashed] (6.3,.2)--(8.3,.2);
  \draw [below] (7.3,.2) node {$(B \o{B})^4$};
  \draw [line width=1.2] (.4,.5) to [out=50, in=130] (9,.5);
  \draw [line width=1.2] (3,.5) to [out=60, in=120] (3.4,.5);
  \draw [line width=1.2] (6,.5) to [out=70, in=110] (8.6,.5);
  \draw [above] (8.6,0) node {$A$};
  \draw [above] (9,0) node {$\o{A}$};
 \end{tikzpicture}

 \caption{All $A$-matchings of $A(\o{B} B)^4\o{B}A\o{A}(B\o{B})^6B\o{A}(B\o{B})^4 A \o{A}$ described in Example~\ref{ex-superset2} which can be extended in $C_5 C_{11}$ ways.}
 \label{fig-A-match-ex-2}
\end{figure}

We are now ready to state a theorem that defines a superset for $\cR(n,1)$ in terms of Catalan numbers.

\begin{theorem}\label{thm-Rnm-superset}
For a positive integer $n$, 
 every value in $\cR(n,1)$ can be expressed as $\sum_{r=1}^{h} p_r$, where each $p_r$ is of the form $\prod_{i = 1}^j  C_{s_i}$ with $0\leq h \leq C_t$ for some $0\leq t \leq \frac{n}{2}$ and  $n-t = s_1 + \cdots + s_j$ for some $1\leq j \leq t+1$.
\end{theorem}

\begin{proof}
Again, it suffices to consider $w \in \hP(n,2)$. 
Without loss of generality, assume $A$ occurs $t$ times in $w$ with $1\leq t \leq \frac{n}{2}$. 
Let $\ell_1,\ldots, \ell_m$ be the lengths of the maximal $B$-subwords of $w$.

Fix a non-crossing $A$-matching $\varphi$ of $w$. 
There are at most $C_t$ such matchings. 
Then $\varphi$ induces a grouping of the maximal $B$-subwords and thus a partition of the multiset $\{\ell_1, \ell_2, \ldots, \ell_m\}$ into multisets $S_1, S_2, \ldots, S_q$. 

Let $r_i$ be the sum of the elements in $S_i$. 
We claim each $r_i$ is even. 
This follows from the fact that the number of letters between any $A$ and $\o{A}$ is even and all $A$ and $\o{A}$ letters are paired in a non-crossing perfect matching. 

Therefore, we can write $r_i = 2s_i$ for some $s_i \in\mathbb{N}$. 
Moreover, the $B$-subwords left to pair together by the $A$-matching group to have the form $(B\o{B})^{s_i}$ or $(\o{B}B)^{s_i}$ and thus can be matched in $C_{{s_i}}$ ways.

Since there are $t$ letters which are $A$, the arcs from the $A$-matching partition the $B$-subwords into at most $t+1$ groups with lengths $2s_1, 2s_2, \ldots, 2s_j$ where $j \leq t+1$ and $\sum_{i=1}^j s_i = n-t$. 
The number of ways to extend the $A$-matching to a non-crossing matching of $w$ is $\prod_{i=1}^j C_{s_i}$.
\end{proof}

Computing the set described in Theorem~\ref{thm-Rnm-superset}, we have the following supersets of $\cR(n,m)$ where the ellipses indicate that all integers in that range are present in the set.
\[
\begin{array}{l c l}
    \cR(1,1) & \subseteq & \{0, 1\}; \\
    \cR(2,1) & \subseteq & \{0, 1, 2\}; \\
    \cR(3,1) & \subseteq & \{0, 1, 2, 5\}; \\
    \cR(4,1) & \subseteq & \{0, \ldots, 5, 14\}; \\
    \cR(5,1) & \subseteq & \{0, \ldots, 7, 10, 14, 42\}; \\
    \cR(6,1) & \subseteq & \{0, \ldots, 22, 25, 28, 42, 132\}; \\
    \cR(7,1) & \subseteq & \{0, \ldots, 52, 56, 57, 58, 60, 61, 70, 84, 132, 429\}; \\
    \cR(8,1) & \subseteq & \{0, \ldots, 178,  182, 183, 184, 186, 187, 196, 210, 264, 429, 1430\}. 
\end{array}
\]
 
The following corollary establishes the largest five values in $\mathcal{R}(n,1)$. Consequently, any integer between $C_{n-2}+C_{n-3}$ and $C_n$ is only in $\mathcal{R}(n,1)$ if it is one of these five values. 
\begin{corollary} \label{cor-Rn1-gaps}
 For $n \geq 13$,
 \[
  \cR(n,1) \subseteq \{0, \ldots, C_{n-2} + C_2\cdot C_{n-4}\} \cup \{C_{n-2} + C_{n-3}, C_3\cdot C_{n-3}, C_2\cdot C_{n-2}, C_{n-1}, C_n\}.
 \]
\end{corollary}
\begin{proof}
For integers $n$ and $t$ with $0 \leq t \leq \floor{n/2}$, 
 let $Y(n,t)$ be the set of all integers of the form $\prod C_{b_i}$ with $b_1 + b_2 + ... + b_j = n-t$ and $1 \leq j \leq t+1$. 
For example, $Y(n,0) = \{ C_n \}$ and
\[
Y(n,1) = \left\{C_{n-1}C_0, C_{n-2}C_{1}, C_{n-3}C_{2}, \ldots, C_{\ceil{\frac{n-1}{2}}}C_{\floor{\frac{n-1}{2}}} \right\}.
\]
Then let $Z(n,t)$ be the set of all $h$-term sums of elements (with repetition) from $Y(n,t)$ with $0 \leq h \leq C_t$. 
Theorem~\ref{thm-Rnm-superset} says $\cR(n,1) \subseteq Z(n,0) \cup \cdots \cup Z(n, \ceil{n/2})$.

One convenient property of Catalan numbers is $C_{n-j}C_j < C_{n-i}C_i$ for $0\leq i < j \leq n/2$. 
As a consequence, $C_{n-t}C_0$ is the maximum value in $Y(n,t)$ and $C_{t}(C_{n-t}C_0)$ is the maximum value in  $Z(n,t)$.

Now note that $Z(n,0)=Y(n,0)$ and $Z(n,1)=Y(n,1)$ which are given above. 
When $t=2$, $C_2(C_{n-2}C_0)=2C_{n-2}$ is the maximum value in $Z(n,2)$. 
The next three largest values in $Z(n,2)$ are 
\[C_{n-2}C_{0} + C_{n-3}C_{1}> C_{n-3}C_0 + C_{n-3}C_0 = 2C_{n-3} > C_{n-3}C_1 + C_2C_{n-4},\] 
 all of which are less than $C_3C_{n-3}$. 

Since the largest two values in $Y(n,3)$ are $C_{n-3}C_0$ and $C_{n-4}C_1$, 
 the maximum value in $Z(n,3)$ is $C_3C_{n-3}$ while the next largest value is $4C_{n-3}+C_{n-4}$ which is smaller than $C_{n-2} + C_2C_{n-4}$.

Finally, for $4\leq t \leq n/2$, all values in $Z(n,t)$ are at most $C_4C_{n-4}$ which is less than $C_{n-2} + 2 C_{n-4}$.
\end{proof}

Although Theorem~\ref{thm-Rnm-superset} only defines a superset for $\cR(n,1)$, we show that the largest gaps stated in Corollary~\ref{cor-Rn1-gaps} do hold for $\cR(n,m)$ in general. 

\begin{proposition}
For $n \geq 13$,
\[
\cR(n,m) \subseteq \{0, 1, \ldots, C_{n-2} + C_2\cdot C_{n-4}\} \cup \{C_{n-2} + C_{n-3}, C_3\cdot C_{n-3}, C_2\cdot C_{n-2}, C_{n-1}, C_n\}.
\]
\end{proposition}

\begin{proof}
Fix a $k$-foldable word $w \in \hP(n,2m)$. 
Let $n_i$ be the number of occurrences of $A_i$ in $w$. 
First note that $k \leq \prod_{i=1}^{2m} C_{n_i}$. This is tight if $w=(A_1\o{A}_1)^{n_1} (A_2\o{A}_2)^{n_2} \ldots (A_{2m}\o{A}_{2m})^{n_{2m}}$. 
Since $C_t \geq C_i C_{t-i}$ for $0\leq i \leq t/2$, if $w$ has at least $3$ letters and their complements, then $k\leq C_1C_1C_{n-2}$ which is less than $C_{n-2}+C_2 C_{n-4}$. 
So if $k \geq C_{n-2}+C_2C_{n-4}$, then $w \in \hP(n,2)$ and the result follows from Corollary~\ref{cor-Rn1-gaps}. 
\end{proof}

Having described a nontrivial superset for $\cR(n,m)$, we turn our attention to finding values that are contained in $\cR(n,1)$ and hence in $\cR(n,m)$. 
The next proposition verifies that the largest five values in the Corollary~\ref{cor-Rn1-gaps} superset truly are in $\cR(n,1)$.

\begin{proposition}
For $n\geq 3$,
\[
\{C_{n-2}+C_{n-3}, C_3C_{n-3}, C_2C_{n-2},  C_{n-1}, C_n  \}\subseteq \cR(n,1).
\]
\label{prop-large_gaps}
\end{proposition}
\begin{proof}
The word $(A\o{A})^t (\o{A}A)^{n-t} \in \cP(n,1)$ has the same number of non-crossing perfect matchings as the corresponding word $(A\o{A})^t (B\o{B})^{n-t} \in \hP(n,2)$, 
 so we have $C_tC_{n-t}\in \cR(n,1)$ for $0\leq t \leq n$. 
It only remains to show that $C_{n-2}+C_{n-3}\in \cR(n,1)$ when $n\geq 3$. 
Consider the words
\[
w = (A\o{A})^{n -3}\,\o{A}AA\o{A}\o{A}A \quad \text{ and } \quad \hat{w}= (A\o{A})^{n -3}\,B\o{B}A\o{A}B\o{B},
\]
which have the same number of non-crossing perfect matchings. 
There are only two $B$-matchings in $\hat{w}$. 
When each $B$ is matched with the $\o{B}$ that immediately follows it, there are $C_{n-2}$ ways to extend this to a non-crossing perfect matching of $\hat{w}$. 
The other $B$-matching can be extended in only $C_{n-3}$ ways.  
So $\hat{w}$, and hence $w$, has $C_{n-2}+C_{n-3}$ non-crossing perfect matchings.
\end{proof}

The Catalan numbers provided a way to describe the largest values in  $\cR$. 
Before discussing the smallest values, we use the Catalan numbers once more to establish a family of values in $\cR$. 

\begin{proposition} \label{prop-jCn}
For non-negative integers $j$ and $\ell$ that satisfy $2j\leq n-\ell$,
\[(j+1) C_{\ell} \in \cR(n,1).\]
\end{proposition}
\begin{proof}
Consider the word $w_{j,\ell} = (\o{A}A)^{\ell}A^{n-\ell - j}\o{A} \,^j A^j\o{A}\,^{n-\ell-j} \in \cS(n,1)$.

Each $A$ in the prefix $(\o{A}A)^{\ell}$ must match with an $\o{A}$ also in the prefix. 
Otherwise, the number of $A$'s and $\o{A} 's$ between the matched pair will not be equal. 
There are $C_{\ell}$ ways to define a non-crossing matching on $(\o{A}A)^{\ell}$.

The remainder of the word 
$A^{n-\ell-j}\o{A}\,^j A^j \o{A}\,^{n-\ell-j}$
matches in $(j+1)$ ways since $j<n-\ell-j$. 
It can be folded onto a path of length $n-\ell$ rooted at one end. 
Then after making the only Type~2 move available, one more available move is created. 
After $j$ of these moves, we will have explored the space of all foldings.
Thus $w_{j,\ell}$ folds in $(j+1)C_{\ell}$ ways.
\end{proof}

\subsection{Small Values in $\cR(n,1)$}

In search of the smallest value which is not in $\cR(n,1)$, we find that (aside from $3 \not\in\cR(3,1)$) all integers $i\leq n$ are in $\cR(n,1)$.

\begin{proposition}\label{prop-small}
	If $n\geq 4$, then $\{0,1,\ldots,n\}\subseteq \cR(n,1)$.
\end{proposition}

\begin{proof}
	First notice that $A^{2n}$ is not foldable and $A^n \o{A}\,^n$ is $1$-foldable for any $n$. 
    Also, $A^{n-2}\o{A}\,^2A^2\o{A}\,^{n-2}$ is $3$-foldable when $n\geq 4$. 
    Thus we have $\{0,1,3\} \subseteq \cR(n,1)$ for $n \geq 4$. 

	Let $n$ be a positive integer. 
    Consider the word $w_\ell=\o{A}A^{\ell}\o{A}\,^{j}A^{j}\o{A}\,^{\ell}A$ where $1 \leq \ell < n$ and $j=n-1-\ell$.  
    We will attain the even values in the interval $[2,n]$ when $j < \ell$ and the odd values in the interval $[5,n]$ when $\ell \leq j$.

	If $j < \ell$, that is $0\leq j\leq\frac{n-2}{2}$, we find that $2j+2\in \cR(n,1)$. 
    To see this, observe that $w_{\ell}$ folds around both trees in Figure~\ref{fig-small} provided $\alpha\in \{0,1,\ldots, j\}$.
    Thus $w_{\ell}$ is at least $(2j+2)$-foldable. 
    From the tree on the left, there is a Type~2 move that will transform the tree into the one on the right. 
    At the vertex of degree $4$, there is a Type~1 move (if $\alpha>0$) and a Type~2 move (if $\alpha<j$), each of which results in a tree of the same description with a different value of $\alpha$. 
    A similar argument can be made for the tree on the right. 
    As the graph $G_{w_{\ell}}$ of $w_{\ell}$-valid trees is connected, $w_{\ell}$ is exactly $(2j+2)$-foldable. 
	

	If $\ell \leq j$, that is $1\leq \ell <\frac{n}{2}$, we find that $2\ell+3\in \cR(n,1)$. First observe that $w_{\ell}[1]$ must bond with $w_{\ell}[2]$, $w_{\ell}[2n]$ or $w_{\ell}[2j+2]$ based on the subwords created. 
    In the case that $w_{\ell}[1]$ bonds with $w_{\ell}[2n]$, the subword $A^{\ell}\o{A}\,^{j}A^{j}\o{A}\,^{\ell}$ folds in exactly $\ell+1$ ways (as on the left tree in Figure~\ref{fig-small}). 
    In the case that $w_{\ell}[1]$ bonds with $w_{\ell}[2]$, then $w_{\ell}[2n]$ bonds with either $w_{\ell}[2n-1]$ or $w_{\ell}[2\ell+1]$, the first extends to $\ell$ different foldings of $w_{\ell}$ and the second extends in only one way. 
    Finally, if $w_{\ell}[1]$ bonds with $w_{\ell}[2j+2]$, there is again only one way to extend this to a folding of $w_{\ell}$. 
    Thus $w_{\ell}$ is $(2\ell+3)$-foldable. 
	
\begin{figure}[ht!]

  \begin{tikzpicture}
  \coordinate (bottom) at (0,-1);
  \coordinate (top) at (0,5);
  \coordinate (mid) at (0,2);
  \coordinate (left) at (-3,1);
  \coordinate (right) at (3,1);
  \foreach \x in {1,2,3,4,5} \path (bottom)--(mid) coordinate [pos=.20*\x] (b\x) {};
  \foreach \x in {1,2,3,4,5} \path (mid)--(top) coordinate [pos=.20*\x] (t\x) {};
  \foreach \x in {1,2,3,4} \path (left)--(mid) coordinate [pos=.25*\x] (l\x) {};
  \foreach \x in {1,2,3,4} \path (right)--(mid) coordinate [pos=.25*\x] (r\x) {};

  \draw (top) node[above] {\textit{\small{root}}};
  \draw (bottom)--(b2);
  \draw (b3)--(mid);
  \draw (mid)-- (t1) (t2)--(top);
  \path (b2)--(b3) node [pos=.7] {$\vdots$};
  \path (t1)--(t2) node [pos=.7] {$\vdots$};
  \foreach \x in {1,2,3,4,5} \draw [fill=black] (b\x) circle (0.06);
  \foreach \x in {1,2,3,4,5} \draw [fill=black] (t\x) circle (0.06);
  \draw [fill=black] (bottom) circle (0.06);

 \draw (left)--(l1) (l2)--(mid)--(r2) (r1)--(right);
  \path  (l1) -- (l2) node [midway, sloped] {$\cdots$};
  \path  (r1) -- (r2) node [midway, sloped] {$\cdots$};
  \foreach \x in {1,2,3,4} \draw [fill=black] (l\x) circle (0.06);
  \foreach \x in {1,2,3,4} \draw [fill=black] (r\x) circle (0.06);
  \draw [fill=black] (left) circle (0.06);
  \draw [fill=black] (right) circle (0.06);
  
     \begin{scope}[every node/.style={scale=.7}]

   \foreach \x in {1,2,4,5} \path (bottom)--(mid) node [pos = .20*\x-.1, left=-.05] {$\overline{A}$};
      \foreach \x in {1,2,4,5} \path (bottom)--(mid) node [pos = .20*\x-.1, right=-.05, yshift=-1] {${A}$};
         \foreach \x in {1,3,4} \path (mid)--(top) node [pos = .20*\x-.1, left=-.05, yshift=-1] {${A}$};
      \foreach \x in {1,3,4} \path (mid)--(top) node [pos = .20*\x-.1, right=-.05] {$\overline{A}$};
      \foreach \x in {5} \path (mid)--(top) node [pos = .20*\x-.1, left=-.05] {$\overline{A}$};
      \foreach \x in {5} \path (mid)--(top) node [pos = .20*\x-.1, right=-.05, yshift=-1] {${A}$};
      
      \foreach \x in {1,3,4} \path (left)--(mid) node [pos = .25*\x - .15, above=0] {$A$};
        \foreach \x in {1,3,4} \path (left)--(mid) node [pos = .25*\x - .15, below=0] {$\overline{A}$};   
              \foreach \x in {1,3,4} \path (right)--(mid) node [pos = .25*\x - .15, above=0] {$\overline{A}$};
        \foreach \x in {1,3,4} \path (right)--(mid) node [pos = .25*\x - .15, below=0] {${A}$};   
\end{scope}
  
\path (left)--(mid) node [pos=.95] (left1) {};
  \draw [decoration={brace}, decoration={raise=2ex}, decorate] (left)--  (left1) node [midway,above=.35, sloped] {$\alpha$};
\path (right)--(mid) node [pos=.95] (right1) {};
  \draw [decoration={brace}, decoration={raise=2ex}, decorate] (right1)--(right) node [midway,above=.35, sloped] {$\alpha$};
\end{tikzpicture}
 \hspace{.4in}
\begin{tikzpicture}
  \coordinate (bottom) at (0,-1);
  \coordinate (top) at (0,5);
  \coordinate (mid) at (0,2);
  \coordinate (left) at (-3,1);
  \coordinate (right) at (3,1);
  \foreach \x in {1,2,3,4,5} \path (bottom)--(mid) coordinate [pos=.20*\x] (b\x) {};
  \foreach \x in {1,2,3,4,5} \path (mid)--(top) coordinate [pos=.20*\x] (t\x) {};
  \foreach \x in {1,2,3,4} \path (left)--(mid) coordinate [pos=.25*\x] (l\x) {};
  \foreach \x in {1,2,3,4} \path (right)--(mid) coordinate [pos=.25*\x] (r\x) {};
  
  \coordinate (lt) at (-.75,3.5);
  \coordinate (rt) at (.75,3.5);
  \draw (lt)--(t3)--(rt);
  \foreach \x in {lt,rt} \draw[fill=black]  (\x) circle (0.06); 
  
  \draw (0,4) node[above] {\textit{\small{root}}};
  \draw (bottom)--(b2);
  \draw (b3)--(mid);
  \draw (mid)-- (t1) (t2)--(t3);
  \path (b2)--(b3) node [pos=.7] {$\vdots$};
  \path (t1)--(t2) node [pos=.7] {$\vdots$};
  \foreach \x in {1,2,3,4,5} \draw [fill=black] (b\x) circle (0.06);
  \foreach \x in {1,2,3} \draw [fill=black] (t\x) circle (0.06);
  \draw [fill=black] (bottom) circle (0.06);

 \draw (left)--(l1) (l2)--(mid)--(r2) (r1)--(right);
  \path  (l1) -- (l2) node [midway, sloped] {$\cdots$};
  \path  (r1) -- (r2) node [midway, sloped] {$\cdots$};
  \foreach \x in {1,2,3,4} \draw [fill=black] (l\x) circle (0.06);
  \foreach \x in {1,2,3,4} \draw [fill=black] (r\x) circle (0.06);
  \draw [fill=black] (left) circle (0.06);
  \draw [fill=black] (right) circle (0.06);
  
\path (left)--(mid) node [pos=.95] (left1) {};
  \draw [decoration={brace}, decoration={raise=2ex}, decorate] (left)--  (left1) node [midway,above=.35, sloped] {$\alpha$};
\path (right)--(mid) node [pos=.95] (right1) {};
  \draw [decoration={brace}, decoration={raise=2ex}, decorate] (right1)--(right) node [midway,above=.35, sloped] {$\alpha$};
  
   \begin{scope}[every node/.style={scale=.7}]
   \path (lt)--(top) node [pos = .4, above=-.45] {$\overline{A}$};
   \path (lt)--(top) node [pos = .4, below=.47] {${A}$};
   \path (rt)--(top) node [pos = .4, above=-.45] {${A}$};
   \path (rt)--(top) node [pos = .4, below=.47] {$\overline{A}$};
 
   \foreach \x in {1,2,4,5} \path (bottom)--(mid) node [pos = .20*\x-.1, left=-.05] {$\overline{A}$};
      \foreach \x in {1,2,4,5} \path (bottom)--(mid) node [pos = .20*\x-.1, right=-.05, yshift=-1] {${A}$};
         \foreach \x in {1,3} \path (mid)--(top) node [pos = .20*\x-.1, left=-.05, yshift=-1] {${A}$};
      \foreach \x in {1,3} \path (mid)--(top) node [pos = .20*\x-.1, right=-.05] {$\overline{A}$};
      
      \foreach \x in {1,3,4} \path (left)--(mid) node [pos = .25*\x - .15, above=0] {$A$};
        \foreach \x in {1,3,4} \path (left)--(mid) node [pos = .25*\x - .15, below=0] {$\overline{A}$};   
              \foreach \x in {1,3,4} \path (right)--(mid) node [pos = .25*\x - .15, above=0] {$\overline{A}$};
        \foreach \x in {1,3,4} \path (right)--(mid) node [pos = .25*\x - .15, below=0] {${A}$};   
\end{scope}
  
\end{tikzpicture}
 \caption{The general $w_{\ell}$-valid trees for $w_{\ell}=\o{A}A^{\ell}\o{A}\,^{j}A^{j}\o{A}\,^{\ell}A$ in the proof of Proposition~\ref{prop-small}.}
 \label{fig-small}
\end{figure}
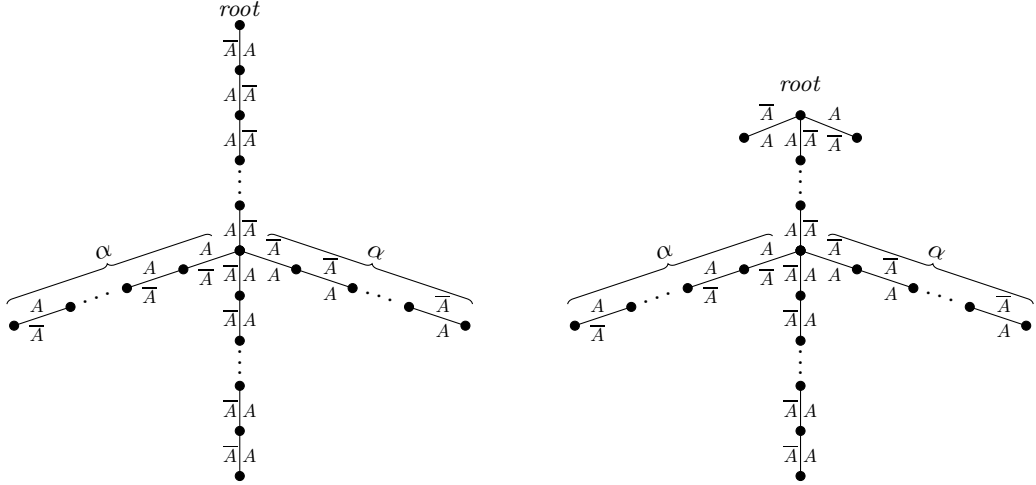
	
\end{proof}

\section{Conclusion}
There are still many questions to be answered regarding the sets $\mathcal{S}_k(n,m)$ and $\mathcal{R}(n,m)$. 
For example, the number of $2$-foldable words, $|\mathcal{S}_2(n,m)|$, is not known.
A complete description of $\mathcal{R}(n,m)$ remains to be determined.
We are particularly interested in the smallest value $k$ for which $k\not\in \mathcal{R}(n,m)$.

\section*{Acknowledgements}

We would like to thank Gavin King, Vaclav Kotesovec, and Derrick Stolee for many helpful conversations. 
All authors were supported in part by NSF-DMS grant \#1500662, ``The 2015 Rocky Mountain-Great Plains Graduate Research Workshop in Combinatorics.'' 
Smith was also supported in part by NSF-DMS grant \#1344199. 

Computations in this paper were performed with Maple\texttrademark{} 2016.2 and SageMath. 
Maple is a trademark of Waterloo Maple Inc. 
Algolib was developed by the Algorithms Project at INRIA. 
Sage code was executed in CoCalc by SageMath, Inc.

\bibliography{bib}
\bibliographystyle{plain}

\end{document}